\documentclass[12pt,reqno]{amsart}
\usepackage{graphicx}
\usepackage{psfrag}
\usepackage{pxfonts}
\usepackage{mathrsfs}
\usepackage{color}
\usepackage{amsmath,amssymb}

\numberwithin{equation}{section}

\theoremstyle{plain}
\newtheorem{maintheorem}{Theorem}

\newcommand{\Z}{\mathbb{Z}}

\newtheorem{theorem}{Theorem}[section]

\newtheorem{corollary}[theorem]{Corollary}
\newtheorem{proposition}[theorem]{Proposition}
\newtheorem{lemma}[theorem]{Lemma}
\newtheorem{definition}[theorem]{Definition}

\newtheorem{example}{Example}

\theoremstyle{remark}

\begin{document}

\thanks{This research has no funding. The authors thank the referees for the careful reading and the numerous useful comments and suggestions.}

\author{J. S. C. Costa}
\address{DEMAT-UFMA S\~{a}o Lu\'{i}s-MA, Brazil.}
\email{jsc.costa@ufma.br}

\author[F. Micena]{F. Micena}
\address{
  IMC-UNIFEI Itajub\'{a}-MG, Brazil.}
\email{fpmicena82@unifei.edu.br}


\renewcommand{\subjclassname}{\textup{2000} Mathematics Subject Classification}

\date{\today}

\setcounter{tocdepth}{2}

\title{Some Generic Properties of Partially Hyperbolic Endomorphisms}
\maketitle
\begin{abstract}
In this work, we deal with a notion of partially hyperbolic endomorphism. We explore topological properties of this definition and we obtain, among other results, obstructions to get center leaf conjugacy with the linear part, for a class of partially hyperbolic endomorphism $C^1-$sufficiently close to a hyperbolic linear endomorphism. Indeed such obstructions are related to the number of center directions of a point. We provide examples illustrating these obstructions.  We show that for a manifold $M$ with dimension $n \geq 3,$ admitting a non-invertible partially hyperbolic endomorphisms, there is a  $C^1$ open and dense subset $\mathcal{U}$ of all partially hyperbolic endomorphisms with degree $d \geq n,$ such that any $f \in \mathcal{U}$ is neither $c$ nor $u$ special.
\end{abstract}

\section{Introduction and Statements of the Results}\label{section.preliminaries}

In $1970s,$ the works \cite{PRZ} and \cite{MP75}  generalized the notion of Anosov diffeomorphism for non-invertible maps, introducing the notion of Anosov endomorphism. We consider $M$ a $C^{\infty}$ closed manifold. 

Let $f: M \rightarrow M$ be a  $C^1$ local diffeomorphism. We say that $f$ is an Anosov endomorphism if for every $(x_n)_{n \in \mathbb{Z}},$ an $f-$orbit, this is $f(x_n)=x_{n+1}$,  there is a splitting

$$T_{x_i} M = E^s_f(x_i) \oplus E^u_f(x_i), \forall i \in \mathbb{Z},$$
such that $Df$ is uniformly expanding on $E^u_f$ and uniformly contracting on $E^s_f.$

Anosov endomorphisms satisfy some peculiar properties that give them characteristics quite different from the Anosov diffeomorphism. As  definition above, given $f: M \rightarrow M$ an Anosov endomorphism and $ \bar{x} = (x_n)_{n \in \mathbb{Z}}, \bar{y} = (y_n)_{n \in \mathbb{Z}} $ two different orbits for $f$ such that $x_0 = y_0,$ it is possible that $E^u_f(x_0) \neq E^u_f(y_0),$ where $E^u_f(x_0)$ is the unstable direction defined by $\bar{x} $ and $E^u_f(y_0)$ is the unstable direction defined by $\bar{y}.$ By \cite{PRZ}, these directions are integrable to unstable local discs, so at a point $x$ we can have more than one local unstable manifold. This fact is an obstruction to the structural stability of Anosov endomorphism, see \cite{PRZ} or \cite{MP75}.  In fact, in \cite{MT16}, the authors showed that for a transitive Anosov endomorphism, either it is special (at every point is defined a unique unstable direction) or for a residual set $R \subset M,$ every point in $R$ has defined infinitely many unstable directions.

Recently, partially hyperbolic endomorphisms are being studied, for example, see \cite{HH21} and \cite{HE19}. Here we present and explore a kind of partially hyperbolic
endomorphism. In fact, this concept is a natural generalization of the concept of Anosov endomorphism. Roughly speaking, a local diffeomorphism $f: M \rightarrow M$ is considered a partially hyperbolic endomorphism if for any orbit $\bar{x} = (x_n)_{n \in \mathbb{Z}},$ we have the decomposition

$$T_{x_i} M = E^s_f(x_i) \oplus E^c_f(x_i) \oplus  E^u_f(x_i), \forall i \in \mathbb{Z},$$
where $Df$ is uniformly expanding on $E^u_f,$ uniformly contracting on $E^s_f$ and $E^c_f$ is a direction with intermediate behavior. The details we will give in the
Section 2.

It is not to hard to see that for any point we have defined a unique $E^s_f$ direction.
\begin{definition}
Let $f:M\rightarrow M$ be a partially hyperbolic endomorphism. We say that $f$ is $u-$special if for all $x\in M,$ the unstable direction $E^u_f(x)\subset T_xM$ is independent of the orbit $\bar{x} = (x_n)_{n \in \mathbb{Z}},$ such that $x_0 = x.$ Similarly, we define the concepts of a such $f$ being $c-$special and $cu-$special.
\end{definition}

\begin{definition}
Let $f:M\rightarrow M$ be a partially hyperbolic endomorphism. We say that $f$ is special if it is simultaneously $c$ and $u-$special.
\end{definition}

An Anosov automorphism $A: \mathbb{T}^3 \rightarrow \mathbb{T}^3$ with three eigenvalues $0 <  |\lambda_1| < 1 <  |\lambda_2| < |\lambda_3| $ and $|\lambda_1 \cdot \lambda_2 \cdot \lambda_3 | > 1,$ is special.

It is well known that for partially hyperbolic diffeomorphisms,
there are foliations $W^s_f$ and $W^u_f$ tangent to $E^s_f$ and $E^u_f,$ respectively, but the distribution $E^c_f$ may not be integrable, for instance see section 6.1 of \cite{pesin2004lectures}.
If $E^c_f$ is  one-dimensional, then it is integrable, but not necessarily uniquely integrable (see \cite{HHU2}). In \cite{B},  M. Brin shows that for (absolute) partially hyperbolic diffeomorphisms for which the stable and unstable foliations have a geometrical condition (quasi-isometry), then $E^c_f$ is uniquely integrable, that is, there
is a foliation $W^c_f$ tangent to $E^c_f.$ The existence of a center foliation is the corn of the dynamical coherence property. As in our definition partially hyperbolic endomorphisms are local diffeomorphisms, then for all orbit $\bar{x} = (x_n)_{n \in \mathbb{Z}},$ there is a stable and unstable manifolds at $x_0.$ Analogous questions arise concerning dynamical coherence in the endomorphism setting.

For all $f: M \rightarrow M$  is  possible to define the space \mbox{$M^f=\{(x_n)_{n \in \mathbb{Z}};\, x_n\in M, f(x_n)=x_{n+1}\}$}, endowed with the metric $\bar{d}$

$$
\bar{d}(\bar{x},\bar{y})=\displaystyle\sum_{i\in {\Z}}\frac{d(x_i,y_i)}{2^{|i|}},
$$
where $d$ is the Riemannian metric in $M$. The space $M^f$ is called the inverse limit space of $M$ with respect to $f.$  If $f$ is  continuous and $M$ is compact, then $(M^f, \bar{d})$ is compact. Given $\bar{x} = (x_n)_{n \in \Z},$ we say that $x_n$ is the $n-$coordinate of $\bar{x} .$ In $M^f$ we can define a natural extension of $f,$ denoted by $\bar{f}:M^f \rightarrow M^f, $ such that, given $\bar{x} = (x_n)_{n \in \Z}$ we have $\bar{f}(\bar{x}) = \bar{y} = (y_n)_{n \in \Z},$ where $y_{n} = x_{n+1}.$ Since $f$ is continuous it is possible to prove that $\bar{f}$ is continuous.

Analogous to Anosov endomorphisms,  we can prove that there are partially hyperbolic endomorphisms $f: M \rightarrow M$ admitting points $x\in M$ for which there are defined several unstable and central directions. Consequently, may there be several unstable and possibly many center manifolds.

In the present work, our main results are the followings.

\begin{maintheorem}\label{Teo C}
Let  $f:M\rightarrow M$ be a transitive partially hyperbolic endomorphism. For each $\sigma \in \{u,c\}$  the following dichotomy holds:
\begin{itemize}
\item Either $f$ is a $\sigma-$special partially hyperbolic endomorphism.
\item Or there exists a residual subset $\mathcal{R}\subset M,$ such that for every $x \in \mathcal{R},$ $x$ has
infinitely many $\sigma$-directions.
\end{itemize}

\end{maintheorem}

It is important to note that the above Theorem \ref{Teo C} is independent of the dimension of the center bundle.

In the next theorem, we exhibit classes of partially hyperbolic endomorphisms satisfying  Theorem \ref{Teo C} and in some sense, these examples are $C^1-$residual.
\begin{maintheorem}\label{Teo D}  Let $M$ be a  smooth closed manifold with dimension $n \geq 3,$ admitting a non-invertible partially hyperbolic endomorphisms. Then there is a  $C^1$ open and dense subset $\mathcal{U}$ of all partially hyperbolic endomorphisms with degree $d \geq n,$ such that any $f \in \mathcal{U}$ is neither $c$ nor $u$ special.
\end{maintheorem}

An interesting point is related to the study of the concept of leaf conjugacy for endomorphisms.\begin{definition} Given $f, g : M \rightarrow M$ partially hyperbolic endomorphisms. We say that $f$ and $g$ are center leaf conjugated if for both are defined center foliations $\mathcal{F}^c_f$ and $\mathcal{F}^c_g$ respectively and there is a homeomorphism $h: M \rightarrow M$ such that
$$h(\mathcal{F}_f^c(x)) = \mathcal{F}_g^c(h(x))\; \mbox{and}\; h(f(\mathcal{F}^c_f(x)) = g(h(\mathcal{F}^c_f(x))),$$
that is, $h$ sends center leaves of $f$ to center leaves of $g$ and preserves the dynamics of the center leaves space.
\end{definition}

From the above definition $f$ and $g$ are $c-$special.

Consider $A: \mathbb{R}^3 \rightarrow \mathbb{R}^3 $ a linear map whose the matrix has integer entries. In this way, $A$ induces a map in $\mathbb{T}^3,$ that we denote also by $A:\mathbb{T}^3 \rightarrow \mathbb{T}^3 .$ Suppose that $A$ has three \mbox{eigenvalues} $\lambda_i,$ $ i =1,2,3$ such that $0 < |\lambda_1| < 1 < |\lambda_2| < |\lambda_3|,$ we can consider the corresponding eigenspaces $E_{\lambda_i}, i=1,2,3.$ For $A:\mathbb{T}^3 \rightarrow \mathbb{T}^3 $ we can consider $E^s_A = E_{\lambda_1}, E^{wu}_A = E_{\lambda_2}$ and $E^{uu}_A = E_{\lambda_3.}$ We can interpret $A$ as a special partially hyperbolic of $\mathbb{T}^3,$ with partially hyperbolic decomposition $E^s_A \oplus E^c_A \oplus E^u_A,$ where $ E^s_A = E_{\lambda_1},$ \mbox{$ E^c_A = E^{wu}_A = E_{\lambda_2},$} and $ E^u_A = E^{uu}_A = E_{\lambda_3}. $ The bundles $E^{wu}_A$ and $E^{uu}_A$ are named respectively weak and strong unstable bundles.

\begin{corollary}\label{Teo B}
 Let $A: \mathbb{T}^3 \rightarrow \mathbb{T}^3  $ be a linear Anosov endomorphism, with degree bigger than two, with partially hyperbolic decomposition
 $E^s_A \oplus E^{wu}_A \oplus E^{uu}_A.$ Given any $\varepsilon > 0,$ there is $f: \mathbb{T}^3 \rightarrow \mathbb{T}^3,$ a partially hyperbolic endomorphism $\varepsilon-C^1-$close to $A,$ for which there is a residual set $R \subset \mathbb{T}^3,$ such that, for any $x \in R$ is defined infinitely many $wu$ and $uu$ bundles.  Additionally, for each $f$ as before,  given $x \in \mathbb{T}^3$ and a bundle $E^{wu}_f(x),$ there is a submanifold $\mathcal{F}^{wu}_f(x)$ tangent to a direction $E^{wu}_f$ in each point, and at the point $x,$ $\mathcal{F}^{wu}_f(x)$ is tangent to the chosen $E^{wu}_f(x).$  Particularly, $f$ is not center leaf conjugated to $A.$
\end{corollary}

We remark that a  partially hyperbolic endomorphism $f$ being non $c-$special is an obstruction to leaf conjugacy with its linearization $A.$  From Theorem \ref{Teo D}, the property to be non $c-$special is a $C^1-$generic characteristic for partially hyperbolic endomorphisms with degree bigger than two in dimension three. This property, in particular, implies that leaf conjugacy is not ensured, even $C^1-$close linear Anosov automorphism as above. Corollary \ref{Teo B} is quite different of Theorem B in \cite{HH21}, that establishes that if a   partially hyperbolic $f: \mathbb{T}^2 \rightarrow \mathbb{T}^2 $ endomorphisms has hyperbolic linearization $A,$  then $f$ and $A$ are center leaf conjugated.

\begin{corollary}\label{Teo E}
 The $C^1$interior of the set of all special partially hyperbolic endomorphisms is exactly the set of all partially hyperbolic diffeomorphisms.
\end{corollary}

In our theorems we are considering $\dim(E^{\sigma}_f) \geq 1,$ for any $\sigma \in \{s,c,u\}.$



\section{Basic Preliminaries}

We start this section with the precise definition of partially hyperbolic endomorphism.

\begin{definition}\label{def 1}
Let $f:M\rightarrow M$ be a $C^1$ local diffeomorphism. We say that $f$ is a \emph{partially hyperbolic endomorphism} if there is a Riemannian metric
$\langle\cdot,\cdot\rangle$ and
constants \mbox{$0<\nu<\gamma_1\leq\gamma_2<\mu$} with $\nu<1,$ $\mu>1$ and $C>1$ such
that for each orbit $(x_n)_{n\in{\Z}}$ of $f,$ this is $f(x_n)=x_{n+1},$ there is a decomposition
$$T_{x_n}M=E^s_f({x_n})\oplus E^c_f(x_n)\oplus E^u_f(x_n)$$ satisfying:

\begin{enumerate}

\item $Df_{x_i}(E^{\ast}_f(x_i))=E^{\ast}_f(x_{i+1}), \ast \in \{s,c,u\},$  for any $i \in \mathbb{Z},$
\item $||Df^n_{x_i}(v^s)||\leq C\nu^n||v^s||,$
 \item  $C^{-1}\gamma^n_1||v^c||\leq||Df^n_{x_i}(v^c)||\leq C\gamma^n_2||v^c||,$
 \item $C^{-1}\mu^n||v^u||\leq||Df^n_{x_i}(v^u)||,$

\end{enumerate}
for any $n \geq 0,$  $i \in \mathbb{Z},$  $ v^s\in E_{x_i}^s,  v^c\in E_{x_i}^c$ and  $v^u\in E_{x_i}^u.$
\end{definition}

This definition is also called absolute partially hyperbolic endomorphism, in contrast with another weaker definition where the constants \mbox{$0<\nu<\gamma_1\leq\gamma_2<\mu$} depend on $x.$ Here we will only consider absolute partially hyperbolic endomorphism.

\begin{example}
The map $A: \mathbb{T}^2 \rightarrow \mathbb{T}^2  $ induced by the matrix
$A=
\left(
  \begin{array}{cc}
    n & 1 \\
    1 & 1 \\
  \end{array}
\right)$
com $n\geq 2,$ is an Anosov endomorphism, in this way,
$f: \mathbb{T}^3 \rightarrow \mathbb{T}^3, $ induced by  $ L=
\left(
  \begin{array}{cc}
    A & 0 \\
     0  & 1 \\
  \end{array}
\right)$
is an example of partially hyperbolic endomorphism.
\end{example}

Now we will show some classical properties, whose proofs are very similar to those of diffeomorphisms, but we write for the completeness of the text.

\begin{proposition}[Uniqueness]\label{uniqueness}
Let $f:M \rightarrow M$ be a partially hyperbolic endomorphism of $M.$ The orbits  $\bar{x} = (x_i)_{i \in \mathbb{Z}}\in M^f$ define uniquely the partially hyperbolic decomposition $E^s_f(x_i) \oplus E^c_f(x_i) \oplus  E^u_f(x_i), i \in \mathbb{Z}.$
\end{proposition}

\begin{proof}
We provide the proof for the center direction, the other cases can be treated similarly. Suppose by contradiction, that for a given partially hyperbolic endomorphism $f:M\rightarrow M,$ there is an orbit $\bar{x}, x_0 = x,$ for which there are  $E^c_1(x)$ and $E^c_2(x)$  two center directions with $E^c_1(x) \neq E^c_2(x).$ We can suppose that there is $v\in E^c_2(x)\setminus E^c_1(x),$ then for the decomposition $E^u_f(x)\oplus E^c_1(x)\oplus E^s_f(x),$ we can write $v=v^s+v^c+v^u,$ with $v^u\neq 0$ or $v^s\neq 0.$

If $v^u\neq 0,$ $v$ grows as $v^u$  by positive iterates of $Df,$ contradicting the fact that $v\in E^c_2(x).$ Analogously, if $v^s\neq 0,$  then taking inverses on $\bar{x}$ it implies that $Df^{-n} v$ grows as $Df^{-n}v^s,$ again a contradiction.
\end{proof}

\begin{proposition}[Continuity]\label{prop 07}
Let $f:M \rightarrow M$ be a partially hyperbolic endomorphism, the spaces $E^u_f(x_0),E^c_f(x_0)$ and $E^s_f(x_0)$ vary continuously in $M^f.$ This is, given any sequence $\bar{x}_k\in M^f$ converging to $\bar{x} = (x_i)_{i \in \mathbb{Z}},$ then $E^{\sigma}_f(x_{0,k})$ converges to $E^{\sigma}_f(x_0),$ with $\sigma\in \{u,c,s\},$ where $x_{0,k}$ is the term in zero coordinate of $\bar{x}_k.$
\end{proposition}

\begin{proof} Unless to take a subsequence we can consider the constants $u,c$ and $s$ the respectively dimensions of $E^u_f,E^c$ and $E^s_f,$ for each $x_{0,k}$ of $\bar{x}_k.$ For each $k$ consider $\{e^k_1,e^k_2,\ldots,e^k_u\},$
$\{v^k_1,v^k_2,\ldots,v^k_c\}$ and $\{w^k_1,w^k_2,\ldots,w^k_s\}$ the respective orthonormal basis of the subspaces $E^u_f(x_{0,k}),$
$E^c_f(x_{0,k})$ and $E^s_f(x_{0,k}).$ By compactness we can admit that these basis converge, respectively to $\{e_1,e_2,\ldots,e_u\},$
$\{v_1,v_2,\ldots,v_c\}$ and $\{w_1,w_2,\ldots,w_s\},$  which are orthonormal basis of the subspaces $F^u(x_0),$
$F^c(x_0)$ and $F^s(x_0).$

By continuity of $Df$ the spaces $F^u(x_0),$
$F^c(x_0)$ and $F^s(x_0)$ satisfies, respectively, the conditions (2), (3) and (4)
of the Definition \ref{def 1}. By uniqueness of the partially hyperbolic decomposition in $M^f$ we get $F^{\sigma}(x_0)=E^{\sigma}_f(x_0),\sigma \in \{s,c,u\}.$ In this way, each  convergent subsequence  of $E^{\sigma}_f(x_{0,k})$ must to converge $F^{\sigma}(x_0).$

\end{proof}

The previous proof was done for the zero coordinate, but the same can be done for any coordinate of $\bar{x}.$

\begin{proposition}[Transversality]
For each $\bar{x} \in M^f,$ the angles between any two direction $E^s_f(x_0),$ $E^c_f(x_0)$ and $E^u_f(x_0)$ are uniformly far from zero. Other cases are treated similarly.
\end{proposition}

\begin{proof} To fix the ideas, consider the stable and unstable bundles. Consider $\theta: M^f \rightarrow [0, \frac{\pi}{2}]$ the angle function $\theta(\bar{x}) = \angle ( E^s_f(x_0), E^u_f(x_0) ).$ Since $E^s_f(x_0)$ and $E^u_f(x_0)$ vary continuously with $\bar{x} \in M^f,$ the angle so is. Since $\theta(\bar{x}) > 0,$ for any $\bar{x} \in M^f,$ by compactness the minimum of $\theta$ is positive. Thus the angle between $E^s_f(x_0)$ and $E^u_f(x_0)$ is uniformly far from zero.

\end{proof}

\begin{lemma}\label{obs 2}
Let $f:M\rightarrow M$ be a partially hyperbolic endomorphism, fix $x\in M,$ $E^c_f(x),$ a center direction defined for $x.$ If $F_f^c(x)$  is another  center direction for $x,$ then
$F_f^c(x)\subset E^c_f(x)\oplus E^s_f(x).$
In other words $E^{cs}_f = E^c_f \oplus E^s_f$ is uniquely defined, for each $x \in M.$
\end{lemma}

\begin{figure}[!htb]
\centering
\includegraphics[scale=0.7]{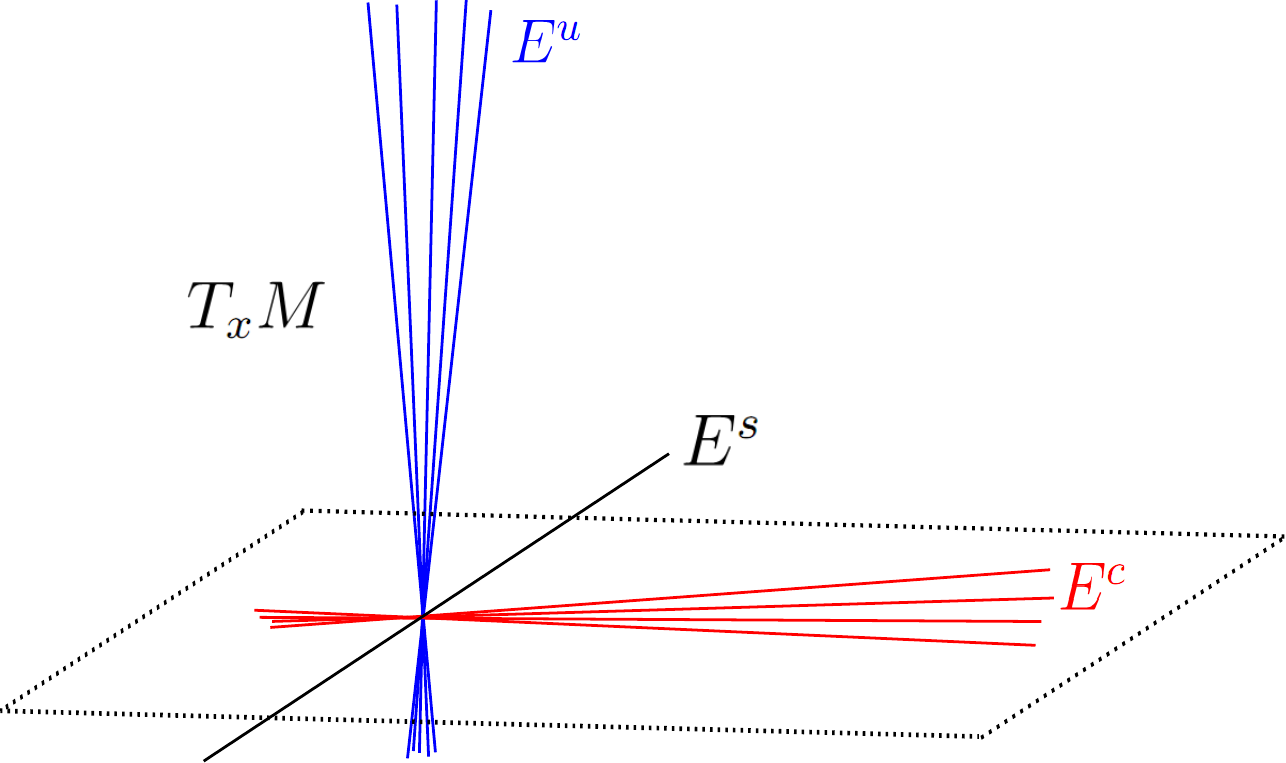}
\caption{Uniqueness of $E^{cs}.$}\label{fig3}
\end{figure}

\begin{proof}
Suppose that for $x \in M$ there is $F_f^c(x)\nsubseteq E^c_f(x)\oplus E^s_f(x),$ then for there is a non null vector $v^c \in F^c_f(x)$ such that  $v^c = v^u + v^{cs},$ where
$v^u\in E^u_f(x)\setminus\{0\},$ and $v^{cs}\in E^c_f(x)\oplus E^s_f(x).$ It  implies that $$||Df_{x_0}^n(v^u)||\leq ||Df_{x_0}^n(v^{cs})||+||Df^n_{x_0}(v^{c})||,$$ this is a contradiction.

\end{proof}

\begin{proposition}\label{prop angulo-zero}
Let $f:M\rightarrow M$ be a partially hyperbolic endomorphism. Given $E^u_1(x),$ $E^u_2(x)$ two unstable directions and $E^c_1(x),$ $E^c_2(x)$ two center direction in the point $x,$ then

(I) The angle between the spaces $Df^n_x(E^c_1(x))$ and $Df_x^n(E^c_2(x))$ converge to zero as $n\rightarrow +\infty.$

(II) The angle between the spaces  $Df^n_x(E^u_1(x))$ and $Df_x^n(E^u_2(x))$ converge to zero as $n\rightarrow +\infty.$

\end{proposition}

\begin{proof}
Let us prove for the case $(I).$ Consider $E^c_1$ and $E^c_2$ two different center directions defined on the same point $x.$ Take any vectors $v_i  \in E^c_i, i =1,2$ non null vectors. By Lemma \ref{obs 2}, there is a vector $v_s \in E^s,$ such that

$$v_1 - v_2 = v^s.$$
If $v^s = 0,$ there is nothing to prove. Suppose that $v^s \neq 0.$ Consider \mbox{$0 <\varepsilon_0  =  \min \{||v_1||, ||v_2||\}.$}  By definition of partially hyperbolic endomorphism,

$$|| \frac{1}{\gamma_1^n}  [Df^n_x \cdot (v_1 - v_2)]|| =  \frac{1}{\gamma_1^n}  ||[Df^n_x \cdot (v^s)]|| \leq C (\frac{\nu}{\gamma_1})^n \rightarrow 0, $$
when $n \rightarrow +\infty.$ Since $|| \frac{1}{\gamma_1^n}  [Df^n_x \cdot (v_i)]|| \geq C^{-1}\varepsilon_0, i = 1,2.$ Also $$|| \frac{1}{\gamma_1^n}  [Df^n_x \cdot (v_1)]|| - || \frac{1}{\gamma_1^n}  [Df^n_x \cdot (v_2)]|| \rightarrow 0,$$ when  $n \rightarrow +\infty,$ we conclude that the angle between $ Df^n_x \cdot v_1$ and  $ Df^n_x \cdot v_2$ converges to $0.$ Since $v_1$ and $v_2$ are taken arbitrary, the angle between the spaces $Df^n(E^c_1(x))$ and $Df^n(E^c_2(x))$ converge to zero as $n\rightarrow\infty.$

\end{proof}

\begin{proposition}[Adapted Metrics]\label{apt_metric} Let $f$ be a partially hyperbolic endomorphism. Then there is a Riemannian metric
$\langle\cdot,\cdot\rangle'$ and constants  $0<\nu'<\gamma_1'\leq\gamma_2'<\mu'$ with $\nu'<1$ and $\mu'>1$ such that for each $f$-orbit $(x_n)_{n\in{\Z}}$, holds

\begin{itemize}
 \item[(1)] $||Df_{x_n}(v^s)||'\leq \nu'||v^s||',\,\,\,\,\,\,\,\,\,\,\,\,\,\,\,\,\,\,\,\,\,\,\,\,\,\,\,\,\,\,\,\,\,\,\,\,\,\,\,\,\,\,\,\,\forall\,\, v^s\in E_{x_n}^s,$
 \item[(2)] $\gamma'_1||v^c||'\leq||Df_{x_n}(v^c)||'\leq \gamma'_2||v^c||',\,\,\,\,\,\,\,\,\,\,\,\,\,\,\,\,\,\,\forall\,\,\, v^c\in E_{x_n}^c,$
 \item[(3)] $\mu'||v^u||'\leq||Df_{x_n}(v^u)||',\,\,\,\,\,\,\,\,\,\,\,\,\,\,\,\,\,\,\,\,\,\,\,\,\,\,\,\,\,\,\,\,\,\,\,\,\,\,\,\,\,\,\forall\,\, v^u\in E_{x_n}^u.$
\end{itemize}
\end{proposition}

\begin{proof}
Let \mbox{$0<\nu<\gamma_1\leq \gamma_2<\mu$} and $C>1$ be the constants as in the Definition \ref{def 1} and consider $N$ big enough such that $C\nu^N< 1$ and $(C^{-1}\mu^N)^2>  1.$
For $v_1,v_2\in T_xM,$ we define the metric
$$
\langle v_1,v_2\rangle'_x=\displaystyle\sum_{j=0}^{N-1}\langle Df^j_{x_0}(v_1),Df^j_{x_0}(v_2)\rangle_{x_j}.
$$
Since $M$ is a compact manifold and $f$ is $C^1$ there is a
constant $K > 1$ such that for all $x\in M$ and $v\in T_xM$,  $\langle v,v\rangle\leq\langle v,v\rangle'\leq K\langle v,v\rangle.$ If $v^s\in E^s_f(x_0),$ then
\begin{align*}
||Df_{x_0}(v^s)||'&=\sqrt{\displaystyle\sum_{j=0}^{N-1}\langle Df^{j+1}_{x_{0}}(v^s),Df^{j+1}_{x_{0}}(v^s)\rangle_{x_{j+1}}}\\
&=\sqrt{\langle v^s,v^s\rangle'_{x_0}+\langle Df^N_{x_0}(v^s),Df^N_{x_0}(v^s)\rangle_{x_N}-\langle v^s,v^s\rangle_{x_0}}\\
&\leq\sqrt{\langle v^s,v^s\rangle'_{x_0}+[(C\nu^N)^2-1]\langle v^s,v^s\rangle_{x_0}}\\
&\leq C\nu^N ||v^s||'.
\end{align*}

We can take $\nu'= C\nu^N.$ Now  If $v^u\in E^u_f(x_0),$
\begin{align*}
||Df_{x_0}(v^u)||'&=\sqrt{\displaystyle\sum_{j=0}^{N-1}\langle Df^{j+1}_{x_{0}}(v^u),Df^{j+1}_{x_{0}}(v^u)\rangle_{x_{j+1}}} &\geq\sqrt{\left[1+\frac{(C^{-1}\mu^N)^2-1}{K}\right]}||v^u||'.
\end{align*}


Then we can take  $\mu'=\sqrt{\left[1+\frac{(C^{-1}\mu^N)^2-1}{K}\right]}.$ Consider $v^c\in E^c_f(x_0),$ we can get
\begin{align*}
\sqrt{\left[1+\frac{(C\gamma_1^N)^2-1}{K}\right]}||v^c|| \leq ||Df_{x_0}(v^c)||\leq C\gamma_2^N||v^c||'.
\end{align*}

We can take $\gamma_1'= \sqrt{\left[1+\frac{(C\gamma_1^N)^2-1}{K}\right]}$ and $\gamma_2'= C\gamma_2^N.$ If it is necessary, we increase $N$ to ensure that $0<\nu'<\gamma_1'\leq\gamma_2'<\mu'$ with $\nu'<1$ and $\mu'>1,$ as required.

\end{proof}

\begin{proposition}[\cite{MP75}]\label{prop MP1}
If $f: M \rightarrow M$ is a  local diffeomorphism, then at the level of universal cover the lift
$\widetilde{f}:\widetilde{M}\rightarrow\widetilde{M}$ is indeed a diffeomorphism.
\end{proposition}

\begin{proposition}\label{pro 6}
Let $f:M\rightarrow M$ be a $C^1$ local diffeomorphism then, $f$ be a partially hyperbolic endomorphism if and only if the lift $\widetilde{f}$  defined in the universal cover is a partially hyperbolic diffeomorphism.
\end{proposition}


\begin{proof}

Denote by $\pi: \widetilde{M} \rightarrow M$ the natural projection such that $\pi(\tilde{x}) = x,$ a $C^{\infty}$ cover map. The fundamental property says us that $\pi \circ \widetilde{f} = f \circ \pi.$ In this way, a bi infinite orbit of $\widetilde{f}$ projects to a bi infinite orbit of $f.$

Consider $\Pi \subset M^f$ the subset constituted by all bi infinite orbits which are projected of bi infinite orbits of $\widetilde{f}.$

Consider $\bar{x} = (x_n)_{n \in \mathbb{Z}} \in \Pi.$ By partial hyperbolicity, consider $E^{\sigma}_{f}(x_i), i \in \mathbb{Z}, \,\mbox{with}\, \sigma \in\{s,c,u\}.$ Let $(y_i)_{i \in \mathbb{Z}}$ be an orbit of $\widetilde{f}$ such that $\pi(y_n) = x_n, n \in \mathbb{Z}.$ For each integer $n,$ consider $E^{\sigma}_f(y_n) = D\pi^{-1}_{x_n} E^{\sigma}_f(x_n),$ where $\pi^{-1}: U_n \rightarrow V_n$ is a $C^{\infty}$ diffeomorphism between open neighborhoods of $x_n$ and $y_n$ respectively.

Since

$$D\pi_{\widetilde{f}(y)} D\widetilde{f}^n_{y} = Df^n_{\pi(y)}D\pi_{y}, $$
and by uniform limitation of $D\pi$ and $D\pi^{-1},$ we obtain $\widetilde{f}$ is partially hyperbolic.

Now suppose that $\widetilde{f}$ is partially hyperbolic. for a given $\bar{x} \in M^f, \bar{x} =  (x_n)_{n \in \mathbb{Z}},$ for each $n < 0,$ consider $y_n \in \widetilde{M}$ such that $\pi (y_n) = x_n.$ Consider $y_0^{(n)}= \widetilde{f}^n(y_n).$ The bi infinite orbit of $y_0^{(n)}$ projects to an orbit $z^{(n)} =  (z^{(n)})_{k \in \mathbb{Z}} \in M^f$ such that $z^{(n)}_k = x_k,$ for all $k \geq n.$ In $M^f $ we obtain $z^{(n)} \rightarrow \bar{x}.$ Unless to take a subsequence we can suppose that $D\pi_{y_0^{(n)}} E^{\sigma}_{\widetilde{f}}(y_0^{(n)}) \rightarrow F^{\sigma}_f(x_0).$ Now fixing $k \geq 0,$ an integer number
$$D\pi_{\widetilde{f}(y)} D\widetilde{f}^k_{y} = Df^k_{\pi(y)}D\pi_{y}, $$
using the uniform limitation of $D\pi$ and $D\pi^{-1},$ and continuity of $Df,$ the bundles $F^s_f(x_0), F^c_f(x_0), F^u_f(x_0),$ satisfies the partially hyperbolic definition of $f$ for $\bar{x}.$

By $D\widetilde{f}$ invariance $F^{\sigma}_f(x_{i}) = \displaystyle\lim_{n  \rightarrow +\infty} D\pi_{y_i^{(n)}} E^{\sigma}_{\widetilde{f}}(y_i^{(n)}),$ where $y_i^{(n)} = \widetilde{f}^{n+i}(y_n),$ for every $i \in \mathbb{Z}.$

Finally, we conclude that $f$ is a partially hyperbolic endomorphism.

\end{proof}

The above proposition allows us to study partially hyperbolic endomorphism by using good properties of partially hyperbolic diffeomorphism. A useful tool in the theory of partially hyperbolic diffeomorphism is cone conditions. In fact, this is a more concrete way to obtain partial hyperbolicity.

For each $\bar{x}\in M^f$ consider $E^u_f(\bar{x}),E^c_f(\bar{x})\subset T_xM,$ where $E^{\sigma}_f(\bar{x})$ denotes $E^{\sigma}_f(x) ,$ as in definition of partial hyperbolicity, and $x$ is the zero coordinate of $\bar{x}.$ Denote by $p:M^f\rightarrow M$ the canonical projection and for each $\sigma\in\{s,u,c\},$ define

$$
\mathcal{E}^{\sigma}_f(x)=\displaystyle\bigcup_{\bar{x}\,\in\, p^{-1}(x)} E^{\sigma}_f(\bar{x}).
$$

The following is a consequence of the construction in Proposition \ref{pro 6}.

\begin{proposition}\label{baixados}
For each $\sigma\in\{s,u,c\},$
$$
\mathcal{E}^{\sigma}_f(x)=\overline{\displaystyle\bigcup_{\pi(y)=x}D\pi_y(E^{\sigma}_{\widetilde{f}}(y))}.
$$
\end{proposition}

 \begin{corollary}
 Let $f: M \rightarrow M$ be a partially hyperbolic endomorphism. Then the dimensions of $E^s_f, E^c_f,$ and $E^u_f$ are constant independent of $\bar{x} \in M^f.$
 \end{corollary}

\begin{proof}
Since $\widetilde{f}: \widetilde{M} \rightarrow \widetilde{M}$ is a partially hyperbolic diffeomorphism, then for $\widetilde{f},$ the dimensions of $E^s_f, E^c_f$ and $E^u_f$ are constant.  By Proposition $\ref{baixados}$  we obtain the result.
\end{proof}

We remark that not every orbit of $f$ can be reached by projection from the universal cover, because each $x \in M,$ the set $\pi^{-1}(\{x\})$ is countable, while the set  of bi infinite orbits of $x$ is non-countable, when $f$ is $n$ to one, $n \geq 2.$

\begin{definition}
Given an orthogonal splitting of the tangent bundle of $M,$ $TM=E\oplus F,$ and a constant $\beta>0,$
for any $x\in M$ we define the cone centered in $E(x)$ with angle $\beta$ as
$$
C(x, E(x), \beta)=\{v\in T_xM:\, ||v_F||\leq\beta ||v_E||,\, \mbox{where}\,v=v_E+v_F, v_E\in E(x), v_F\in F(x)  \}.
$$
\end{definition}

Now let $f:M\rightarrow M$ be a partially hyperbolic diffeomorphism. Fix $x\in M,$ give $\beta>0$ define the families of cones

\begin{align*}
&C^s(x,\beta)=C(x,E^s_f(x), \beta),\,\,\,\,\,\,\,\,\,\,\hspace{3cm} C^u(x,\beta)=  C(x,E^u_f(x), \beta),\\
&C^{cs}(x,\beta)= C(x,E^{s}_f(x)\oplus E^c_f(x), \beta),\,\,\,\,\,\,\,\,\,\,\, C^{cu}(x,\beta)= C(x,E^{c}_f(x)\oplus E^{u}_f(x), \beta).
\end{align*}

It is known that a diffeomorphism $f$ is  partially hyperbolic if, and only if, there is $0<\beta<1$ and constants
\mbox{$0<\nu<\gamma_1\leq\gamma_2<\mu$} with $\nu<1,$ $\mu>1$ and an adapted metric such that for all $x \in M$ we have:
\begin{align*}
Df^{-1}_x(C^{\sigma}(x,\beta))&\subset C^{\sigma}(f^{-1}(x),\beta),\,\, \sigma=s,cs, \nonumber\\
Df_x(C^{\psi}(x,\beta))&\subset C^{\psi}(f(x),\beta),\,\, \psi=u,cu \nonumber
\end{align*}
and
\begin{align*}\label{eq C2}
||Df^{-1}_x v||&>\nu^{-1}||v||,\,\, v\in C^s(x,\beta), \\
||Df^{-1}_x v||&>\gamma^{-1}||v||,\,\, v\in C^{cs}(x,\beta), \\\
||Df_x v||&>\mu||v||,\,\, v\in C^{u}(x,\beta), \\
||Df_x v||&>\gamma_1||v||,\,\, v\in C^{cu}(x,\beta).
\end{align*}

From Proposition \ref{pro 6} and cone condition, we obtain the following.

\begin{proposition}
The set of partially hyperbolic endomorphism is an $C^1$-open set.
\end{proposition}

\begin{proof} Let $f: M \rightarrow M$ be a partially hyperbolic endomorphism. If $g$ is a local diffeomorphism $C^1-$close to $f,$ then in the lifts $\widetilde{f}$ and $\widetilde{g}$ are $C^1$ close in the universal cover $\widetilde{M}.$ If $g$ is enough $C^1-$close to $f,$ then $\widetilde{g}$ is enough close to $\widetilde{f}.$ By cone condition $\widetilde{g}$ is partially hyperbolic and by Proposition \ref{pro 6}, the local diffeomorphism $g$ is a partially hyperbolic endomorphism.

\end{proof}

\section{Proof of Theorem \ref{Teo C}  }
For a given $f:M \rightarrow M$  a partially hyperbolic endomorphism denote by $c(x)$ the number of center directions at $T_xM,$ similarly $u(x)$ denotes de number of unstable directions at $T_xM.$ Using the results of the previous section we are able to prove Theorem $\ref{Teo C}$ in the same lines to Theorem 1.4 of \cite{MT16}.
For the proof, we need the following proposition.

\begin{proposition}
Let $f:M\rightarrow M$ be a transitive partially hyperbolic endomorphism. Then for each $\sigma\in \{c,u\},$
either $f$ is a $\sigma-$special partially hyperbolic endomorphism or there is $x\in M$ such that $\sigma(x)=\infty.$

\end{proposition}

\begin{proof}
Fix $\sigma\in \{c,u\}.$ We assume that $\sigma(x)<\infty$ for all $x\in M.$ Define the sets
$\Lambda_k=\{x\in M; \sigma(x)\leq k\}.$ Using a continuity argument (Proposition \ref{prop 07}), the set  $M\setminus \Lambda_k$ is open, then for each $k$  the set $\Lambda_k$ is closed.
 As
$$
M=\bigcup_{k=1}^{+\infty}\Lambda_k,
$$
by Baire Category Theorem there is $k_0$ such that $int(\Lambda_{k_0})\neq \emptyset.$

We claim that $M=\Lambda_{k_0}.$
To prove the claim, consider any $x\in M$ with $l>0,$  $\sigma-$directions and $V_x$ a neighbourhood
of $x$ such that any point in $V_x$ has at most $l$ $\sigma-$directions. Consider a point
in $V_x$ with dense orbit and take an iterate of it that belongs to $int(\Lambda_{k_0}).$
Since $x \mapsto\sigma(x)$ is non decreasing along the orbits it implies that $l\leq k_0.$ We get

$$
M=\bigcup_{k=1}^{k_0}\Lambda_k
$$
by connectedness of $M$ we conclude $M=\Lambda_{k_0}.$

Now we prove that $M=\Lambda_{1}.$ Suppose that, there is $x\in M$ such that \mbox{$\sigma(x)\geq 2$}
and choose $E^{\sigma}_1(x),$ $E^{\sigma}_2(x)$ two different $\sigma-$directions of $x.$
Let $\theta>0$ be the angle between  $E^{\sigma}_1(x)$ and $E^{\sigma}_2(x).$
Consider $U_x$ a small neighbourhood of $x$ such that for all $y\in U_x$ has at least two $\sigma-$directions.
Consider $E^{\sigma}_1(y)$ and $E^{\sigma}_2(y)$ such that $\angle(E^{\sigma}_1(y),E^{\sigma}_2(y))>\frac{\theta}{2}.$

Let $x_0$ be a point with dense orbit. By denseness there is $n_1$ be a large number satisfying:

\begin{itemize}
\item[(i)] $f^{n_1}(x_0)\in U_x,$
\item[(ii)] $\angle(Df^{n_1}_{x_0}(E),Df^{n_1}_{x_0}(F))<\frac{\theta}{3},$ for any $E,F\in \mathcal{E}^{\sigma}_f(x_0).$
\end{itemize}

By definition of $U_x$ these two properties imply  $\sigma(f^{n_1}(x_0))>\sigma(x_0)+1,$ otherwise there wouldn't be two $\sigma-$directions such that the angle between them it was bigger than $\frac{\theta}{2}.$

Repeating this argument, it is possible to obtain an infinite sequence $f^{n_k}(x_0)$ such that $\sigma(f^{n_{k+1}}(x_0))>\sigma(f^{n_k}(x_0))+1, k =0,1,\ldots,$ where $n_0 = 0,$ it contradicts $M=\Lambda_{k_0}$ with $k_0>1.$

\end{proof}

To finish the proof of the theorem, we need to prove that $\sigma(x)=\infty,$ for a residual
set $\mathcal{R}\subset M,$ whenever $f$ is not $\sigma-$special partially hyperbolic endomorphism.
In fact, suppose that there is $x\in M,$ such that $\sigma(x)=\infty.$ Given $k>0,$ fix exactly $k$
different directions at $x,$ and by continuity of $ \bar{x} \in M^f \mapsto E^{\sigma}_f(\bar{x}),$ consider $U_x^k$ a neighbourhood of $x,$ such that $\sigma(y)\geq k,$ for all $y\in U^k_x.$
Since $f$ is transitive, the open set $V^k=\displaystyle\bigcup_{i\geq 1}f^i(U^k_x)$ is dense in $M.$ Finally,
consider $\mathcal{R}=\displaystyle\bigcap_{k\geq 1}V^k, $ which is a residual set. By construction, given $x\in\mathcal{R}$
we obtain $\sigma(x)\geq k$ for all $k>1,$ which implies $\sigma(x)=\infty.$


\section{Proof of Theorem\ref{Teo D} }
 Evidently there are partially hyperbolic endomorphisms $g: M \rightarrow M,$ such that $deg(g) \geq 3.$ To see this, it is enough $g = f_0^{2},$ where $deg(f_0) \geq 2.$ By the same way, we can obtain partially hyperbolic endomorphisms with degree bigger than $k,$ for any $k \geq 3.$ For simplicity we give the detailed proof for case $n = 3.$

\begin{proof}[Proof of Theorem B] Since $f$ is a covering map there is $\tau > 0$ such that $f(x) = f(y),$ then $d(x, y) \geq \tau.$

Consider $A: M \rightarrow M$ a $c-$special partially hyperbolic endomorphism. Suppose that $deg(A)\geq 3.$ Consider $x$ a non fixed point for $A.$ Since $x$ is not fixed, then any pre-image of $x$ is not fixed too. So the point $x$ has three distinct pre-images, each one of them different from $x.$

Let $x_1, x_2, x_3$ be three distinct pre-images of $x,$ of course  $x_i \neq x,$ $i =1,2,3.$ Choose $0 < r_1, r_2 < \frac{\tau}{3}$ and balls $B_i = B(x_i, r_i), i = 1,2,$ such that $x \notin B_i,$ $ i =1,2.$ By previous lemma $B(x_1, r_1) \cap B(x_2, r_2) = \emptyset.$  Moreover, by previous lemma any $z \in M$ has a pre-image in the complement of $B_1 \cup B_2,$ otherwise, as $deg(A) \geq 3,$ there would be two pre-images of $z$ in the same $B_i,$ it would be a contradiction with the construction of $B_i.$  Thus it is possible construct inverse branches of $x_1, x_2, x_3$ on the complement of  $B_1 \cup B_2.$  For each $i = 1,2,3$ consider $( \ldots, x_{i2}, x_{i1}, x_i)$ a inverse branch of $x_i$ out of $B_1 \cup B_2,$ such that $A(x_{ik}) = x_{ik-1}, k \geq 1$ and $x_{i0} = x_i.$ Let $\bar{x}_i, i=1,2,3$ be the orbit constructed such that in zero coordinate is the point $x_i,$ and the negative coordinates are out of $B_1 \cup B_2.$ We write $E^{cu}_A(x_i)$ to designate the center unstable at $x_i$ corresponding the orbit $\bar{x}_i, i =1,2,3.$

Let $\varphi: M \rightarrow M $ be a $C^{\infty} $ diffeomorphism $\varepsilon-C^1-$close to identity such that

\begin{itemize}
\item $\varphi$ is the identity out of $B_1 \cup B_2,$
\item $\varphi(x_i) = x_i,$
\item $\varphi(B_i) = B_i, i=1,2,$
\item $E^{cu}_i:= D\varphi_{x_i} E^{cu}_A(x_i) \neq E^{cu}_A(x_i), i =1,2.$
\item $DA_{x_1} E^{cu}_1 \cap DA_{x_2}E^{cu}_2 \cap DA_{x_3}E^{cu}_A(x_3)  = \{0\}.$
\end{itemize}

To construct $\varphi$ we can use small rotations in local coordinates, such that they fix $x_i, i=1,2,$ after we glue these rotations with identity out of $B_1 \cup B_2$ by a partition of unity.  Since $\varphi$ can be constructed arbitrarily $C^1$ close to the identity, then $f := A \circ \varphi$ is $\varepsilon-C^1-$close to $A.$

If $f$ is $\varepsilon-C^1-$close to $A,$ then $f$ is partially hyperbolic too. Let $\bar{x}_i$ be the orbit constructed such that in zero coordinate is the point $x_i,$ and the negative coordinates are out of $B_1 \cup B_2.$

By construction, relative to  $\bar{x}_i$ we obtain $E^u_f(x_i) = E^u_A(x_i), i =1,2,3.$ It is because the inverse branches taken of $x_1$ and $x_2$ are totally out of $B_1 \cup B_2,$ so we act only with $A$ in all pre-images. Using the construction of invariant sub bundles via cones considering inverse branches of $x_i, i=1,2,3,$ we obtain the same center unstable sub bundle obtained for $A.$

 Moreover

$$Df_{x_i} \cdot E^{cu}_f({x_i}) = DA_{x_i}  D\varphi_{x_i} E^{cu}_A(x_i) = DA_{x_i} E^{cu}_i, i = 1,2, $$
$$Df_{x_3} \cdot E^{cu}_f({x_3}) =DA_{x_3}  \cdot E^{cu}_A(x_3) = E^{cu}_A(x). $$

 In this way, at $x$ we distinguish three center unstable directions $F_1 = DA_{x_1} E^{cu}_1,$
 \mbox{$F_2 =DA_{x_2} E^{cu}_2$} and  $F_3 = E^{cu}_A(x) $ such that
 $$F_1 \cap F_2 \cap F_3 = \{0\}. $$

%
%

Since  $F_1 \cap F_2 \cap F_3 = \{0\}, $ there is no a non null subspace common to $F_1, F_2$ and $ F_3 .$ Thus, for the point $x$ is defined at least two center directions.  Note that if we had only one center direction defined for $x,$ it would be a non-trivial subspace common to $F_1,F_2$ and $F_3.$ With the same argument we conclude that there are at least two strong unstable directions defined for $x.$ The same argument guarantees that there are at least two unstable directions of $x.$ Thus $f$ is neither $c$ nor $u$ special.

If we start with $A$ a $u-$special partially hyperbolic endomorphism with $deg(A) \geq 3,$ we find $f,$ $C^1-$close to $A,$ such that $f$  is neither $c$ nor $u$ special.  The denseness part of Theorem \ref{Teo D} is proved. Let us prove the openness.

Consider $f$ a non $c-$special partially hyperbolic endomorphism and $x$ a point such that is defined at least two center bundles. Since the lift $\widetilde{f}$ is partially hyperbolic there is two points $x_1$ and $x_2$ such that $\pi(x_1) = \pi(x_2) = x,$ such that for $\widetilde{f}$ is defined  center directions of   $E^c_1$ of $x_1$ and $E^c_2$ of $x_2$ such that they  project on different center bundles of $x.$ By continuity of the center bundles (for partially hyperbolic diffeomorphisms) if $g$ is  $C^1$ sufficiently close to $f$ the $\widetilde{g}$ is $C^1-$close to $\widetilde{f}.$ So, for $\widetilde{g}$  is defined  center directions of   $F^c_1$ of $x_1$ and $F^c_2$ of $x_2$ such that they  project on different center bundles of $x.$ Therefore any $g$ that is $C^1$  sufficiently close of $f$ is not $c-$special. The same argument hods changing non $c-$special by non $u-$special.

When $n \geq 3,$ we perturb a partially hyperbolic endomorphism $A$ with degree at least $n.$ The perturbation from $A$ is done around $n-1$ pre-images of a point $x$ by $A.$ At the level $T_xM$ we can create $n$ mutually distinct $(cu-)$sub bundles with dimension at most $n-1,$ so that the intersection of them is trivial. For the obtained $f$ from this perturbation, the conclusion here is the same one when $n=3.$

\end{proof}

\section{ Applications: Proof of Corollaries \ref{Teo B} and \ref{Teo E}}

Here we present some applications of the main Theorems \ref{Teo C} and \ref{Teo D}.

\begin{proof}[Proof of Corollary \ref{Teo B}]
By Theorem \ref{Teo D}, the endomorphism $A$ can be $C^1$ approximated by partially hyperbolic endomorphisms which are neither $wu$ nor $uu$ special, a such $f,$  $\varepsilon -C^1-$close to $A$ we know that $f$ is transitive by \cite{AH}.

By Theorem \ref{Teo C} for a such $f$ there are residual sets $R_{uu}$ and $R_{wu}$ such that $z \in R_{\sigma}$ then at $z$ is defined infinitely many  $E^{\sigma}_f$ sub bundles, $\sigma \in \{wu, uu\}.$ Take the residual set $R = R_{wu} \cap R_{uu}.$

Given $f$ as above consider $E^c_f = E^{wu}_f$ and $E^u_f = E^{uu}_f$ the directions of partially hyperbolic definition.

Every $f$ as  above is such that for any $E^{cu}_f(x)$ of a given $x,$ is defined a tangent submanifold $W^{cu}_f(x),$ which is tangent to a direction $E^{cu}_f$ in each point. In fact,  $E^{cu}_f(x)$ is an unstable direction of Anosov endomorphism definition and  $W^{cu}_f(x),$ the corresponding unstable manifold, see \cite{PRZ}. As we proved for each point is defined a unique $E^{cs}_f,$ moreover in the cover $\mathbb{R}^3,$ it is integrable to a foliation $W^{cs}_{\widetilde{f}}$ which projects to $W_f^{cs}-$foliation of $f$ in $\mathbb{T}^3,$ by \cite{B}.

If we take $ W^{cu}_f(x) \cap W^{cs}_f(x) = W^c_f(x),$ it is a submanifold tangent to $E^{c}_f(x)$ given by $ E^{cu}_f(x)\cap E^{cs}_f(x).$  Thus each $E^c_f$ at $x$ admits a tangent submanifold $W^c_f(x)$, such that $T_xW^c_f(x) = E^c_f(x)$ and for each $z \in W^c_f(x) \setminus \{x\}, $ holds $T_zW^c_f(x) = E^c_f(z),$ for some $E^c_f(z).$

We conclude that at $x$ are defined infinitely many submanifolds which are tangent to directions $E^c_f$ in each point. It makes impossible a center leaf conjugacy between $f$ and $A.$
\end{proof}

\begin{figure}[!h]
\centerline{
\includegraphics[width=10.8cm]{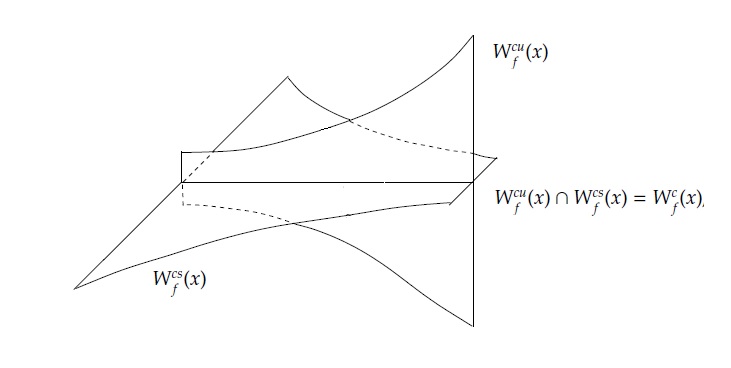}}
\caption{{\small A center leaf at $x$ tangent to a $E^c_f(x)$}.}
\label{figsemi}
\end{figure}

This corollary is a contrast with Theorem B  of \cite{HH21}, which asserts that a partially hyperbolic endomorphism of $\mathbb{T}^2$ with hyperbolic linearization is center leaf conjugated with its linear part.

\begin{proof}[Proof of Corollary \ref{Teo E}]  By similar approach of Theorem \ref{Teo D} we can find $f$ $C^1$ sufficiently close to $A,$ such that $f$ is not $cu-$special. In this way, $A$ is not in the $C^1$ interior of special partially hyperbolic endomorphisms.

\end{proof}

\section{Further comments and questions }

In \cite{PRZ} Przytycki proved that for a given linear Anosov endomorphism of the torus $A: \mathbb{T}^m \rightarrow \mathbb{T}^m ,$  for any fixed $x \in \mathbb{T}^m$ and $\varepsilon > 0, $ there is $f: \mathbb{T}^m \rightarrow \mathbb{T}^m,$ $\varepsilon-C^1$ close to $A,$ such that at $x,$ the map $f$  has defined uncountable unstable directions, in fact, these directions describe a curve in the Grassmannian space. An interesting question is if the points of $\mathcal{R}$ in Theorem \ref{Teo C} are such that for them are defined uncountable $\sigma$ directions. Moreover, in Theorem \ref{Teo C} could be $\mathcal{R} \neq M ?$

For a partially hyperbolic endomorphism $f:M \rightarrow M$  we can also make questions about invariant manifolds tangent to $E^s_f, E^c_f,$ and $E^u_f.$  Suppose that  $\widetilde{f}$ is dynamically coherent. For each $x \in M$ we can consider the center leaves at $x$ by projections of center leaves $\pi(\mathcal{F}^c_{\widetilde{f}}(z)),$ such that $\pi(z) = x.$ All of these projected leaves are homeomorphic one each other? For a dynamically coherent partially hyperbolic endomorphism $f: M \rightarrow M,$ could there be local submanifolds tangent to  $E^c_f$ in each point, such that these submanifolds are not limit of projections of local center leaves of $\widetilde{f}\,?$ Even the study of dynamical coherence and some kind of generalization of leaf conjugacy could be interesting in this setting.

There are a lot of questions that one could explore in this setting. For instance, to understand ergodicity for conservative partially hyperbolic endomorphisms and connections with ergodicity and accessibility in this context. Existence of SRB measures and measures of maximal entropy and their connections with questions involving rigidity, absolute continuity of center foliations (when it makes sense), etc. Recent advances in this setting have been done. For measures of maximal entropy, see \cite{cant} and for absolute continuity of center foliation, see \cite{micnovo}.


\end{document}